\documentclass[11pt,reqno]{amsart}

\usepackage{amssymb,amsmath,amsthm,amscd,latexsym,amsfonts}
\usepackage{mathtools}
\usepackage[T1]{fontenc}

\usepackage{graphicx}
\usepackage{xcolor}
\usepackage{cite}

\usepackage[a4paper,top=3cm, bottom=3cm, left=3cm, right=3cm]{geometry}

\newtheorem{thm}{Theorem}
\newtheorem{defn}{Definition}
\newtheorem{lemma}{Lemma}
\newtheorem{pro}{Proposition}
\newtheorem{rk}{Remark}

\numberwithin{equation}{section} 

\newcommand{\M}{{\mathcal M}}

\newcommand{\bea}{\begin{eqnarray}}
\newcommand{\eea}{\end{eqnarray}}

\newcommand{\Q}{\mathbb{Q}}
\newcommand{\R}{\mathbb{R}}

\def\M{\mathcal M}

\def\R{\mathbb{R}}

\begin{document}
\title [Quadratic stochastic processes]
{Quadratic stochastic processes of type $(\sigma|\mu)$}

\author {B.J. Mamurov, \, U.A. Rozikov, \, S.S. Xudayarov}

\address{B.\ J. \ Mamurov\\ Bukhara State University, The department of Mathematics,
 11, M.Iqbol,  Bukhara, Uzbekistan.}
\email {bmamurov.51@mail.ru}

 \address{U.A.~Rozikov \\ Institute of Mathematics, 81, Mirzo Ulug'bek str., 100125,
Tashkent, Uzbekistan.} \email {rozikovu@yandex.ru}

\address{S.\  \ Xudayarov\\ Bukhara State University, The department of Mathematics,
11, M.Iqbol, Bukhara city.\, Bukhara, Uzbekistan.}
\email {xsanat83@mail.ru}

\begin{abstract} We construct quadratic stochastic processes (QSP) (also known as Markov processes of cubic matrices)
in continuous and discrete times. These are dynamical systems given by (a fixed type,
called $\sigma$) stochastic cubic matrices
satisfying an analogue of Kolmogorov-Chapman equation (KCE) with respect to
a fixed  multiplications (called $\mu$) between
cubic matrices. The existence of a stochastic (at each time)
solution to the KCE provides
the existence of a QSP called a QSP of type $(\sigma | \mu)$.

In this paper, our aim is to construct and study trajectories of QSPs
for specially chosen notions of stochastic cubic matrices
and a wide class of multiplications of such matrices (known as Maksimov's multiplications).
\end{abstract}



\maketitle

\section{Introduction}

%
The Kolmogorov-Chapman equation (KCE) gives the fundamental
relationship between the probability transitions (kernels).
Namely, it is known that (see e.g. \cite{SK}) if
each element of a family of matrices satisfying the KCE
is stochastic, then it generates a Markov process.
%
%
%
%
In this paper following \cite{CLRm} we study Markov process of cubic matrices,
which is a two-parametric family of
cubic stochastic matrices (we fix a notion of stochastic
 matrix and fix a multiplication rule of cubic matrices) satisfying the KCE.
 The main question of this study is to describe dynamics of the
 process given by cubic matrices. This question is very important in the theory
 of dynamical systems to know future evolution of the system.

%
%
%

Let us give necessary definitions and facts.


\subsection{Maksimov's cubic stochastic matrices}
Denote $I=\{1,2,\dots,m\}$. Let $\mathfrak C$ be the set of all $m^3$-dimensional cubic matrices over the field of real numbers \cite{LRl}.
Denote by $E_{ijk}$, $i,j,k\in I$ the basis cubic matrices in $\mathfrak C$.

Following  \cite{Mak} define the following multiplications for basis matrices $E_{ijk}$:
\begin{equation}\label{6ma0}
 E_{ijk}*_0 E_{lnr}=\delta_{kl}\delta_{jn}E_{ijr},
 \end{equation}
 where $\delta_{kl}$ is the Kronecker symbol.

Then for any two cubic matrices $A=(a_{ijk}), B=(b_{ijk})\in \mathfrak C$
 the matrix $A*_0 B=(c_{ijk})$ is defined by
 \begin{equation}\label{6AB0}
 c_{ijr}=\sum_{k=1}^m a_{ijk}b_{kjr}.
 \end{equation}

 The following results of this section are proven in \cite{Mak} (see also \cite{Rpd} for detailed proofs)

\begin{pro} The algebra of cubic matrices $(\mathfrak C, *_0)$
is a direct sum of algebras of square matrices.
\end{pro}

Define multiplication:
\begin{equation}\label{6ma}
 E_{ijk}*_a E_{lnr}=\delta_{kl}E_{ia(j,n)r},
 \end{equation}
  where $a \colon   I\times I\to I$, $ (j,n) \mapsto a(j,n) \in I$,  is an arbitrary associative binary operation.

  Note that \eqref{6ma0} is  \textsl{not} a particular case of \eqref{6ma}.

  Denote by $\mathcal O_m$ the set of all associative binary operations on $I$.

 The general formula for the multiplication is the extension of \eqref{6ma} by bilinearity, i.e.
 for any two cubic matrices $A=(a_{ijk}), B=(b_{ijk})\in \mathfrak C$
 the matrix $A*_a B=(c_{ijk})$ is defined by
\[
 c_{ijr}=\sum_{l,n: \, a(l,n)=j}\sum_k a_{ilk}b_{knr}.
\]
Note that $c_{ijr}=0$ for $j$ such that $\{l,n: \, a(l,n)=j\}=\emptyset$.


 \begin{lemma} The multiplication (\ref{6ma}) is associative for each associative $a\in \mathcal O_m$.
  \end{lemma}
 If the equation $a(x,u) =v$ (resp. $a(u,x) = v$) is uniquely solvable for any $u, v\in I$
then the operation $a$ on $I$ has right (resp. left) unique solvability.

\begin{lemma}\label{sc} If the operation $a$ on $I$ has right or left unique solvability, then
$$\sum_{d\in I}\sum_{{j,m: \atop a(j,m)=d}}\gamma_{j,m}=\sum_{j\in I}\sum_{m\in I}\gamma_{j,m}.$$
\end{lemma}

\subsection{Stochasticity}
Define several kinds of cubic stochastic matrices (see \cite{Mak,MG}):
a cubic matrix $P=(p_{ijk})_{i,j,k=1}^m$ is called
\begin{itemize}
\item[] $(1,2)$-\emph{stochastic} if
\[p_{ijk}\geq 0, \qquad  \sum_{i,j=1}^mp_{ijk}=1, \ \ \text{for all} \ k.\]

\item[] $(1,3)$-\emph{stochastic} if
\[p_{ijk}\geq 0, \qquad  \sum_{i,k=1}^mp_{ijk}=1, \ \ \text{for all} \ j.\]

\item[] $(2,3)$-\emph{stochastic} if
\[p_{ijk}\geq 0, \qquad  \sum_{j,k=1}^mp_{ijk}=1, \ \ \text{for all} \ i.\]

\item[] $3$-\emph{stochastic} if
\[p_{ijk}\geq 0, \qquad  \sum_{k=1}^mp_{ijk}=1, \ \ \text{for all} \  i,j.\]
The last one can be also given with respect to first and second index.
\end{itemize}

Maksimov \cite{Mak} also defined a twice stochastic matrix:
a (2,3)-stochastic cubic matrix is called \emph{twice stochastic} if
\[\sum_{i=1}^m p_{ijk}=\frac{1}{m}, \qquad \text{for all} \ \ j,k.\]

\begin{pro} $(2,3)$-stochastic (and twice) stochastic
 cubic matrices form a convex semigroup\footnote{ A semigroup is an algebraic structure consisting of a set together with an associative binary operation.} with respect to multiplication (\ref{6ma}).
\end{pro}
%

\begin{rk} One also can show that $(1,2)$-stochastic cubic matrices form a convex semigroup. But
the collection of $(1,3)$-stochastic matrices does
not form a semigroup with respect to multiplication (\ref{6ma}).
\end{rk}

\subsection{Quadratic stochastic processes}

Denote by $\mathcal S$ the set of all possible kinds of stochasticity and
denote by $\mathbb M$ the set
of all possible multiplication rules of cubic matrices.

Let parameters $s\geq 0$, $t\geq 0$, are considered as time.

Denote by
$\M^{[s,t]}=\left(P_{ijk}^{[s,t]}\right)_{i,j,k=1}^{\ \underset{m}{}}$ a cubic matrix with two parameters.
\begin{defn}[\cite{LR}]
A family $\{\M^{[s,t]}: \ s,t\in\R_+\}$ is called a Markov process of cubic matrices (or a quadratic stochastic process (QSP))
of type $(\sigma|\mu)$ if for each time $s$ and $t$ the cubic matrix
$\M^{[s,t]}$ is stochastic in sense $\sigma\in\mathcal S$ and satisfies the Kolmogorov-Chapman equation
(for cubic matrices):
\begin{equation}\label{KC}
\M^{[s,t]}=\M^{[s,\tau]}*_\mu\M^{[\tau,t]}, \qquad \text{for all} \ \ 0\leq s<\tau<t
\end{equation} with respect to the multiplication $\mu\in \mathbb M$.
\end{defn}

QSPs arise naturally in the study of biological and physical systems with interactions (see  \cite{Ke}).
Following \cite{CLRm} assume there are $m$ different types of particles, denoted by $I=\{1,2,\dots,m\}$ the set of all types and our main
aim is to study the asymptotic behavior of the variables
 $\xi^{(t)}_i =$ number of particles of type $i$ at the time $t$.
 The initial state is taken to be fixed and described by $\xi^{(0)}_i$,
  the numbers of particles of type $i\in I$ in the initial (zero) time.
These numbers are assumed finite. Denote by
$$x^{(t)}_i=P^{(t)}(i)={\xi^{(t)}_i\over \sum_{j=1}^m \xi^{(t)}_j},$$
 fraction of particles of type $i$ at the time $t$.
Thus
$$x^{(t)}=(x_1^{(t)}, \dots, x_m^{(t)})\in S^{m-1}=\left\{x=(x_1,\dots, x_m)\in \mathbb R^m: x_i\geq 0, \, \sum_{i=1}^mx_i=1\right\}$$ is a
distribution of the system (i.e. the vector describing fractions of all types of particles) at the moment $t$.

Let $x^{(0)}=(x_1^{(0)}, \dots, x_m^{(0)})\in S^{m-1}$ be an initial distribution on $I$.
For arbitrary moments of time $s\geq 0$ and $t>0$, with $s<t$, the matrix
$\M^{[s,t]}$ gives the transition probabilities
from the distribution $x^{(s)}$ to the distribution $x^{(t)}$.
To use the matrix $\M^{[s,t]}=\left(P_{ijk}^{[s,t]}\right)$
we assume that a particle of type $i\in I$ and a particle of type $j\in I$
have interaction at time $s$, as an interaction process, then
with probability $P_{ijk}^{[s,t]}$ a particle of type $k$ appears at time $t$.
The equation \eqref{KC} gives
the time-dependent evolution law of the interacting process (dynamical system).

Since we should have $x^{(t)}\in S^{m-1}$, one can consider the following models:
\begin{itemize}
\item[-] Consider $P_{ijk}^{[s,t]}$ as the conditional
probability $P^{[s,t]}(k|i, j)$ that $i$th and $j$th particles (physics) or species (biology)
interbred successfully at time $s$, then they produce
an individual $k$ at time $t$.

Assume the ``parents'' $ij$ are independent
for any moment of time $s$, that is
$$P^{(s)}(i,j)=P^{(s)}(i)P^{(s)}(j)=x^{(s)}_ix^{(s)}_j,$$
 and we assume that the matrix $\left(P_{ijk}^{[s,t]}\right)$ is
3-stochastic, then the probability distribution $x^{(t)}$
can be found by the formula of the total probability as
\begin{equation}\label{ks}
x^{(t)}_k=\sum_{i,j=1}^m P^{(s)}(i,j)  P^{[s,t]}(k|i, j)=  \sum_{i,j=1}^mP_{ijk}^{[s,t]}x^{(s)}_ix^{(s)}_j, \ \  k=1, \dots, m, \ \ 0\leq s<t.
\end{equation}
For $1-$stochastic and $2-$stochastic it can be defined similarly, by replacing the corresponding indices.

\item[-] Consider now a physical (biological, chemical) system where
there are $m$ types of ``particles'' or molecules, the set of types is denoted by $I=\{1,\dots,m\}$,
 and each particle may split to two new ones
 having types from $I$. Consider $P_{ijk}^{[s,t]}$ as the conditional
probability $P^{[s,t]}(i,j|k)$ that a particle of type $k$ starts splitting
at time $s$ and finishes splitting at time $t$ and the result is
two particles with $i$th and $j$th types.

Assume  $\left(P_{ijk}^{[s,t]}\right)$ is (1,2)-stochastic then $x^{(t)}$ can be defined by
\begin{equation}\label{xt12}
x^{(t)}_k= \frac{1}{2} \sum_{i,j=1}^m\left(P_{kij}^{[s,t]}+P_{ikj}^{[s,t]}\right)x^{(s)}_j, \qquad k=1, \dots, m, \ \ 0\leq s<t.
\end{equation}
For (1,3)-stochastic and (2,3)-stochastic cases one can define similarly by replacing the indices.
\end{itemize}

Thus finding $P_{ijk}^{[s,t]}$ from the equation \eqref{KC} (at a fixed $(\sigma,\mu)$) and studying the time-dependent behavior
of $P_{ijk}^{[s,t]}$ we can describe the time-dependent evolution of $x^{(t)}$.
%

\subsection{The main problem}
To construct QSPs of type $(\sigma | \mu)$, i.e. to solve \eqref{KC}. To study the dynamics of such system when $t-s\to +\infty$.
 In this paper our aim is to construct and study QSPs, for the Maksimov's multiplication corresponding
to arbitrary operation $a$ on $I$ which has right (or left) unique solvability.

\section{Construction of QSPs}

The equation  \eqref{KC} has the following form
\begin{equation}\label{ms}
P_{ijr}^{[s,t]}=\sum_{l,n: \, a(l,n)=j}\sum_{k=1}^m P^{[s,\tau]}_{ilk}P_{knr}^{[\tau,t]}, \quad \forall i,j,r\in I.
\end{equation}
We have to fix a stochasticity of cubic matrices first and solve (\ref{ms}) in class of such matrices.

Consider Maksimov's multiplication corresponding
to arbitrary operation $a$ on $I$ which has right (resp. left) unique solvability.  Denote

\begin{equation}\label{mq0}q_{ir}^{[s,t]}=\sum_{j=1}^mP_{ijr}^{[s,t]}, \ \  \mathbb Q^{[s,t]}=\left(q_{ir}^{[s,t]}\right)_{i,r=1}^{\ \underset{m}{}}.
\end{equation}
Then using the solvability condition and Lemma \ref{sc} we reduce
equation \eqref{KC} (i.e. (\ref{ms})) to the following
\begin{equation}\label{mq}
q_{ir}^{[s,t]}=\sum_{k=1}^m q^{[s,\tau]}_{ik}q_{kr}^{[\tau,t]}, \quad \forall i, r\in I, \ \ \mbox{i.e.}, \ \
\mathbb Q^{[s,t]}=\mathbb Q^{[s,\tau]}\mathbb Q^{[\tau,t]}.
\end{equation}
Thus the Kolmogorov-Chapman equation for cubic matrices reduced to the Kolmogorov-Chapman equation for square matrices.
Summarizing we have
\begin{pro}\label{as} Any solution of equation \eqref{KC} for a multiplication $\mu$, corresponding
to operation $a$ stifling the condition of Lemma \ref{sc},
can be given by a solution of the system \eqref{mq} with a
matrix  $\mathbb Q^{[s,t]}=\left(q_{ir}^{[s,t]}\right)_{i,r=1}^{\ \underset{m}{}}$ which satisfies
\eqref{mq0}.
\end{pro}

Recall that a square matrix $\mathbb Q=(q_{ij})_{i,j=1}^m$ is called \emph{right stochastic}
if
\[q_{ij}\geq 0, \quad  \forall i,j=1,\dots,m; \qquad \sum_{j=1}^mq_{ij}=1, \ \ \forall i=1,\dots,m.\]

Similarly one can define a \emph{left stochastic} matrix being a non-negative real square matrix, with each column summing to 1 and
a \emph{doubly stochastic} matrix being a square matrix of non-negative real numbers with each row and column summing to 1.

A family of stochastic matrices $\{\mathbb Q^{[s,t]}: s,t\geq 0\}$
is called a \emph{Markov process} if it satisfies the Kolmogorov-Chapman equation (\ref{mq}).

The full set of solutions to (\ref{mq}) is not known yet. But there is a very wide class of its solutions
see \cite{CLR,CLRm, LR, LR0, ORT,RM}. One of these known solutions is the following (non-stochastic, time-homogeneous) matrix:
\begin{equation}\label{defo}
\begin{pmatrix}
q_{11}^{[s,t]}& q_{12}^{[s,t]}\\[2mm]
q_{21}^{[s,t]}& q_{22}^{[s,t]}
\end{pmatrix}=\begin{pmatrix}
\cos (t-s)& \sin (t-s)\\
-\sin (t-s)& \cos (t-s)
\end{pmatrix}.
\end{equation}

In \cite{RM} to construct chains of some algebras, for $m=2$,
 a wide class of solutions of \eqref{mq} is presented,
many of them are non-stochastic matrices, in general.
Let us give a list of families of (left, right, doubly) {\rm stochastic}
square matrices (see \cite{RM}), which
satisfy the equation \eqref{mq}, i.e. they generate interesting Markov processes:
\begin{align*}
 \Q_1^{[s,t]}&=\begin{pmatrix}
g(s)& g(s)\\[2mm]
1-g(s) &1-g(s)\\[2mm]
\end{pmatrix}, \  \text{where} \ \ g(s)\in [0,1] \  \text{is an arbitrary function};\\
\Q_2^{[s,t]}&=\frac{1}{2} \begin{pmatrix}
1+\frac{\Psi(t)}{\Psi(s)} & 1-\frac{\Psi(t)}{\Psi(s)}\\[2mm]
1-\frac{\Psi(t)}{\Psi(s)} & 1+\frac{\Psi(t)}{\Psi(s)}\\[2mm]
\end{pmatrix},
\end{align*}
 where $\Psi(t)> 0$ is an arbitrary decreasing function of $t\geq 0$;
  \begin{align*}
  \Q_3^{[s,t]}&=
\begin{cases}
\ \ \begin{pmatrix}
1 & 0\\[2mm]
0 & 1 \\[2mm]
\end{pmatrix}, & \ \ \text{if} \ \ s\leq t<b,\\[4mm]
\frac{1}{2} \begin{pmatrix}
1 & 1\\[2mm]
1 & 1 \\[2mm]
\end{pmatrix},& \ \ \text{if} \ \ t\geq b,\\
\end{cases}, \quad \text{where} \ \ b>0 ;\\[3mm]
 \Q_4^{[s,t]}&=\begin{pmatrix}
1  & 0\\[2mm]
1-\frac{\psi(t)}{\psi(s)} & \frac{\psi(t)}{\psi(s)}\\[2mm]
\end{pmatrix},
\end{align*}
where $\psi(t)>0$ is a decreasing function
of $t\geq 0$;
\begin{align*}
 \Q_5^{[s,t]}&=\begin{pmatrix}
f(t)& 1-f(t)\\[2mm]
f(t) &1-f(t)\\[2mm]
\end{pmatrix}, \ \ \text{where} \ \ f(t)\in [0,1] \ \ \text{is an arbitrary function};\\
 \Q_6^{[s,t]}(\lambda,\mu)&=
\begin{pmatrix}
1-\frac{\lambda-2\mu}{2(\lambda-\mu)}\left(1- \frac{\theta(t)}{\theta(s)} \right)&
\frac{\lambda-2\mu}{ 2(\lambda-\mu)} \left(1-\frac{\theta(t)}{\theta(s)}  \right)\\[2mm]
\frac{\lambda}{2(\lambda-\mu)} \left(1-\frac{\theta(t)}{\theta(s)} \right)&
 1-\frac{\lambda}{2(\lambda-\mu) }\left(1- \frac{\theta(t)}{\theta(s)}\right) \\[2mm]
\end{pmatrix},
\end{align*}
where  $\lambda$, $\mu$ are real parameters  such that $0<2\mu<\lambda$ and $\theta(t)>0$ is an arbitrary decreasing function;
\[\Q_7^{[s,t]}=
\begin{cases}
\begin{pmatrix}
1 & 0\\[2mm]
0 & 1\\[2mm]
\end{pmatrix}, & \ \ \text{if} \ \ s\leq t<a,\\[4mm]
\begin{pmatrix}
g(t) & 1-g(t)\\[2mm]
g(t) & 1-g(t)\\[2mm]
\end{pmatrix}, & \ \ \text{if} \ \ t\geq a,\\
\end{cases} \  \text{where} \ \ g(t)\in [0,1] \  \text{is an arbitrary function}.\]

We note that the matrices $\Q_i^{[s,t]}$, $i=1,\dots,7$, generate interesting usual Markov processes:
some of them independent on time, some depend only on $t$,
but many of them non-homogenously depend on both $s,t$.  Depending on the statistical models
of real-world processes one can choose
parameter functions (i.e. $g$,  $\Psi$, $\psi$, $f$,  $\theta$) and be able then to control
the evolution (with respect to time) of such Markov processes.

The following lemma gives a connection between stochastic matrices.
\begin{lemma}\label{02} The matrix $\mathcal M^{[s,t]}=\left(P_{ijk}^{[s,t]}\right)$, with $P_{ijk}^{[s,t]}\geq 0$,  is
\begin{itemize}
\item (1,2)-stochastic (resp. (2,3)-stochastic) if and only if the corresponding  matrix $\mathbb Q^{[s,t]}$
is left (resp. right) stochastic.
\item (1,3)-stochastic if and only if the corresponding  matrix $\mathbb Q^{[s,t]}$
satisfies $\sum_{i,r=1}^m q^{[s,t]}_{ir}\equiv m$.
\item 1-stochastic (resp. 3-stochastic) if and only if the corresponding  matrix $\mathbb Q^{[s,t]}$
satisfies $\sum_{i=1}^m q^{[s,t]}_{ir}\equiv m$ (resp. $\sum_{r=1}^m q^{[s,t]}_{ir}\equiv m$).
\item 2-stochastic iff $q^{[s,t]}_{ir}\equiv 1$.
\end{itemize}
\end{lemma}
\begin{proof} It is consequence of the equality \eqref{mq0}.
\end{proof}
\begin{pro}\label{ppp} If $m>1$ and the operation $a$ on $I$ has right or left unique solvability, then equation \eqref{KC} does not
 have any solution in class of $i$-stochastic (for any $i=1, 2, 3$) cubic matrices.
\end{pro}
\begin{proof} We prove in case $i=1$ (the cases $i=2,3$ are similar).
Assume there is a solution $\mathcal M^{[s,t]}=\left(P_{ijk}^{[s,t]}\right)$, which is
1-stochastic, i.e.,
$$P_{ijk}^{[s,t]}\geq 0, \ \ \sum_{i=1}^mP_{ijk}^{[s,t]}=1, \ \ \forall j,k, \ \ 0\leq s<t.$$
Then by Lemma \ref{02} for corresponding square matrix $\mathbb Q^{[s,t]}$ we should have  $\sum_{i=1}^m q^{[s,t]}_{ir}\equiv m$.
Moreover, by Proposition \ref{as} the matrix $\mathbb Q^{[s,t]}$ should satisfy (\ref{mq}), which is impossible, because
$$m=\sum_{i=1}^mq_{ir}^{[s,t]}=\sum_{i=1}^m\sum_{k=1}^m q^{[s,\tau]}_{ik}q_{kr}^{[\tau,t]}=\sum_{k=1}^m mq_{kr}^{[\tau,t]}=m^2>m.$$
\end{proof}

\begin{rk} In case $m=1$ the equation \eqref{KC} becomes the following functional equation
\begin{equation}\label{mb}
P^{[s,t]}=P^{[s,\tau]}P^{[\tau,t]},
\end{equation}
where unknown function is $P^{[s,t]}=P_{111}^{[s,t]}$.

This equation is
known as Cantor's second equation which has a
very rich family of solutions:
\begin{itemize}
  \item[(a)] $P^{[s,t]}\equiv 0$;
  \item[(b)] $P^{[s,t]}=\frac{\Phi(t)}{\Phi(s)}$, where $\Phi$ is an arbitrary
function with $\Phi(s)\ne 0$;
   \item[(c)] \[P^{[s,t]}=\begin{cases}
1, & \text{if} \  \ s\leq t<c,\\[2mm]
0, &  \text{if} \ \ t\geq c.\\
\end{cases} \quad \text{where} \ \ c>0.\]
\end{itemize}
\end{rk}
\begin{rk} \begin{itemize}
\item[1)] According to Proposition \ref{ppp} we do not have $i$-stochastic solutions.

\item[2)] As it was mentioned above multiplication of two $(1,3)$-stochastic matrices may be non $(1,3)$-stochastic. Therefore
it is not clear existence of a $(1,3)$-stochastic solution to \eqref{KC} (i.e. (\ref{ms})). Below we shall construct some examples of such solution.

\item[3)] By above mentioned results one can see that $(1,2)$-stochasticity and $(2,3)$-stochasticity play a symmetric role.
Therefore below we find only $(1,2)$-stochastic solutions of (\ref{ms}).
\end{itemize}
\end{rk}

{\bf Condition 1.} For definiteness let us take $I=\{0,1,2,\dots, m-1\}$ as a group
with respect to operation $a$, defined by $a(i,j)=(i+j) ({\rm mod}\, m)$.
Then it is easy to see that $a$ is uniquely solvable.

Under this condition the elements of the matrix $\M^{[s,t]}$ can be renumbered as $\M^{[s,t]}=\left(P_{ijk}^{[s,t]}\right)_{i,j,k=0}^{\ \underset{m-1}{}}$.

For convenience of the writing of this cubic matrix we introduce square matrix
$$ \M_i^{[s,t]}=\left(P_{ijk}^{[s,t]}\right)_{j,k=0}^{\ \underset{m-1}{}}, \ \ i=0, 1, \dots, m-1.$$
Then the cubic matrix can be written as
$$\M^{[s,t]}=\left(\M_0^{[s,t]} | \M_1^{[s,t]} | \dots | \M_{m-1}^{[s,t]}\right).$$
The equation (\ref{ms}) can be written as
\begin{equation}\label{mm}
P_{ijr}^{[s,t]}=\sum_{k=0}^{m-1} \left(\sum_{l=0}^jP^{[s,\tau]}_{ilk}P_{k(j-l)r}^{[\tau,t]}+
\sum_{l=1}^{m-j-1}P^{[s,\tau]}_{i(j+l)k}P_{k(m-l)r}^{[\tau,t]}\right), \quad \forall i,j,r\in I.
\end{equation}

This is a non-linear system of functional equations with $m^3$ unknown two-variable functions $P_{ijr}^{[s,t]}$.
Trivial solution is $P_{ijr}^{[s,t]}={1\over m^2}$, $\forall i,j,r\in I$, $0\leq s<t$.
The analysis of the system (\ref{mm}) is difficult.
Therefore below we shall mainly consider the case $m=2$.

\subsection{$(1,3)$-stochastic solutions}

Now we construct QSPs of type $(13|a)$, where $13$ means $(1,3)$-stochasticity and $a$ means that we are considering multiplication
 (\ref{6ma}).

 For simplicity let us consider the case $m=2$.
Write a cubic matrix $\M^{[s,t]}$ in the following convenient form:
\begin{equation}\label{edd}
\M^{[s,t]}= \begin{pmatrix}
P_{000}^{[s,t]} &P_{001}^{[s,t]}& \vline &P_{100}^{[s,t]} &P_{101}^{[s,t]}\\[3mm]
P_{010}^{[s,t]} &P_{011}^{[s,t]}& \vline &P_{110}^{[s,t]} &P_{111}^{[s,t]}
\end{pmatrix}.
\end{equation}
This matrix generates QSP of type $(13|a)$ iff
\begin{equation}\label{oke}\begin{array}{lll}
P_{000}^{[s,t]}+P_{001}^{[s,t]}+P_{100}^{[s,t]} +P_{101}^{[s,t]}=1,\\[3mm]
P_{010}^{[s,t]} +P_{011}^{[s,t]}+P_{110}^{[s,t]}+P_{111}^{[s,t]}=1.\\[3mm]
P_{i0j}^{[s,t]}+P_{i1j}^{[s,t]}=q^{[s,t]}_{ij}, \ \ i,j=0,1.
\end{array}
\end{equation}
In addition to these conditions by Condition 1 the equation (\ref{KC}) becomes
\begin{equation}\label{ok}
\left\{\begin{array}{ll}
P_{i0j}^{[s,t]}=\sum_{k=0}^1\left(P_{i0k}^{[s,\tau]}P_{k0j}^{[\tau,t]}+P_{i1k}^{[s,\tau]}P_{k1j}^{[\tau,t]}\right)\\[3mm]
P_{i1j}^{[s,t]}=\sum_{k=0}^1\left(P_{i0k}^{[s,\tau]}P_{k1j}^{[\tau,t]}+P_{i1k}^{[s,\tau]}P_{k0j}^{[\tau,t]}\right), \ \ i,j=0,1.
\end{array}\right.
\end{equation}
In general it is difficult to solve the system (\ref{ok}). Let us solve it in class of functions satisfying
$$P_{000}^{[s,t]}=P_{001}^{[s,t]}=P_{100}^{[s,t]}\equiv f(s,t).$$ Assume also that matrix $\mathbb Q^{[s,t]}$ is left \emph{and} right stochastic.

Using these assumptions and (\ref{oke}) from (\ref{ok}) we get
$$f(s,t)=P_{000}^{[s,t]}=P_{000}^{[s,\tau]}P_{000}^{[\tau,t]}+P_{010}^{[s,\tau]}P_{010}^{[\tau,t]}+
P_{001}^{[s,\tau]}P_{100}^{[\tau,t]}+P_{011}^{[s,\tau]}P_{110}^{[\tau,t]}$$
$$=4f(s,\tau)f(\tau,t)-(q_{00}^{[s,\tau]}+q_{01}^{[s,\tau]})f(\tau,t)-(q_{00}^{[\tau,t]}+q_{10}^{[\tau,t]})f(s,\tau)
+q_{00}^{[s,\tau]}q_{00}^{[\tau,t]}+q_{01}^{[s,\tau]}q_{10}^{[\tau,t]}.$$

Since the matrix $\mathbb Q^{[s,t]}$  satisfies (\ref{mq}) we have
$$q_{00}^{[s,\tau]}q_{00}^{[\tau,t]}+q_{01}^{[s,\tau]}q_{10}^{[\tau,t]}=q_{00}^{[s,t]}$$
and since the matrix is left and right stochastic we get
\begin{equation}\label{fq}
f(s,t)=4f(s,\tau)f(\tau,t)-f(\tau,t)-f(s,\tau)
+q_{00}^{[s,t]}.\end{equation}

Note that $f(s,t)\leq 1/3$.

Take the matrix $\mathbb Q^{[s,t]}=\left(\begin{array}{cc}
1/2& 1/2\\
1/2& 1/2\end{array}\right)$. Thus
$q_{00}^{[s,t]}\equiv 1/2$ then denoting $h(s,t)=4f(s,t)-1$ the equation (\ref{fq}) can be
written as
\begin{equation}\label{mb}
h(s,t)=h(s,\tau)h(\tau,t).
\end{equation}
As Cantor's second equation this equation has solutions:
\begin{itemize}
  \item[(a)] $h(s,t)\equiv 0$;
  \item[(b)] $h(s,t)=\frac{\Phi(t)}{\Phi(s)}$, where $\Phi$ is an arbitrary
function with $\Phi(s)\ne 0$;
   \item[(c)] \[h(s,t)=\begin{cases}
1, & \text{if} \  \ s\leq t<c,\\[2mm]
0, &  \text{if} \ \ t\geq c.\\
\end{cases} \quad \text{where} \ \ c>0.\]
\end{itemize}
Now for each solutions we give corresponding QSP:

(a') The functional equation (\ref{fq}) has solutions $f(s,t)\equiv 1/4$. Thus the cubic matrices (independent on time)
\begin{equation}\label{e4}
\M_1^{[s,t]}= \begin{pmatrix}
1/4 &1/4& \vline &1/4 &1/4\\[3mm]
1/4&1/4& \vline &1/4 &1/4
\end{pmatrix}.
\end{equation}
generate a QSP of type  $(13 | a)$.

(b') The functional equation (\ref{fq}) has solutions $f(s,t)={1\over 4}\left(\frac{\Phi(t)}{\Phi(s)}+1\right)$.
Thus the cubic matrices
\begin{equation}\label{eb4}
\M_2^{[s,t]}= \begin{pmatrix}
f(s,t)&f(s,t)& \vline &f(s,t) &1-3f(s,t)\\[3mm]
{1\over 2}-f(s,t)&{1\over 2}-f(s,t)& \vline &{1\over 2}-f(s,t) &3f(s,t)-{1\over 2}
\end{pmatrix}.
\end{equation}
generate a QSP of type $(13 | a)$ iff
\begin{equation}\label{v}{1\over 6}\leq f(s,t)={1\over 4}\left(\frac{\Phi(t)}{\Phi(s)}+1\right)\leq {1\over 3}, \ \ \mbox{i.e.} \ \
-{1\over 3}\leq \frac{\Phi(t)}{\Phi(s)}\leq {1\over 3}.\end{equation}

Note that the condition (\ref{v}) can be satisfied for a function $\Phi$ when time is discrete, i.e., $t\in \mathbb N$.
Then for example we can take $\Phi(n)=3^{-n}$.

(c') In case (c) we have
$$f(s,t)=\left\{\begin{array}{ll}
1/2, \ \ \mbox{if} \ \ s\leq t<c\\[2mm]
1/4, \ \ \mbox{if} \ \  t\geq c.
\end{array}\right.$$
But this does not satisfy condition $f(s,t)\leq 1/3$.

Summarizing we have

\begin{pro} The matrices $\M_1^{[s,t]}$ defined in (\ref{e4}) generate a QSP of type $(13| a)$.
The matrices $\M_2^{[n,m]}$, $n,m\in \mathbb N$, $n<m$ defined in (\ref{eb4}) generate a
discrete-time QSP of type $(13| a)$.
\end{pro}

\subsection{$(1,2)$-stochastic solutions}

Now we construct QSPs of type $(12|a)$, where $12$ means $(1,2)$-stochasticity and $a$ means that we are considering multiplication
 (\ref{6ma}).

 The matrix (\ref{edd}) generates QSP of type $(12|a)$ iff
\begin{equation}\label{ke}\begin{array}{lll}
P_{000}^{[s,t]}+P_{010}^{[s,t]}+P_{100}^{[s,t]} +P_{110}^{[s,t]}=1,\\[3mm]
P_{001}^{[s,t]}+P_{011}^{[s,t]}+P_{101}^{[s,t]} +P_{111}^{[s,t]}=1,\\[3mm]
P_{i0j}^{[s,t]}+P_{i1j}^{[s,t]}=q^{[s,t]}_{ij}, \ \ i,j=0,1.
\end{array}
\end{equation}
Note that for any left stochastic matrix $\mathbb Q^{[s,t]}$ the conditions (\ref{ke}) are satisfied.
Therefore it suffices to solve the equation (\ref{ok}). Assuming
$$P_{000}^{[s,t]}=P_{001}^{[s,t]}=P_{100}^{[s,t]}\equiv g(s,t)$$
we get
$$g(s,t)=P_{000}^{[s,t]}=P_{000}^{[s,\tau]}P_{000}^{[\tau,t]}+P_{010}^{[s,\tau]}P_{010}^{[\tau,t]}+
P_{001}^{[s,\tau]}P_{100}^{[\tau,t]}+P_{011}^{[s,\tau]}P_{110}^{[\tau,t]}$$
\begin{equation}\label{on}
=4g(s,\tau)g(\tau,t)-(q_{00}^{[s,\tau]}+q_{01}^{[s,\tau]})g(\tau,t)-(q_{00}^{[\tau,t]}+q_{10}^{[\tau,t]})g(s,\tau)
+q_{00}^{[s,\tau]}q_{00}^{[\tau,t]}+q_{01}^{[s,\tau]}q_{10}^{[\tau,t]}.
\end{equation}

Take the matrix $\mathbb Q^{[s,t]}=\left(\begin{array}{cc}
1/2& 1/2\\
1/2& 1/2\end{array}\right)$.  In this case the equation
(\ref{on}) has solutions as (a')-(c').

Now for each solutions we give corresponding QSP:

(a") The functional equation (\ref{on}) has solutions $g(s,t)\equiv 1/4$. Thus $(1,2)$-stochastic  cubic matrix has the form
\begin{equation}\label{u1}
\M^{[s,t]}= \begin{pmatrix}
1/4 &1/4& \vline &1/4 &F(s,t)\\[3mm]
1/4&1/4& \vline &1/4 &1/2-F(s,t)
\end{pmatrix}.
\end{equation}
where $F(s,t)$ should satisfy the equation
$$F(s,t)=P_{101}^{[s,t]}=P_{100}^{[s,\tau]}P_{001}^{[\tau,t]}+P_{110}^{[s,\tau]}P_{011}^{[\tau,t]}+
P_{101}^{[s,\tau]}P_{101}^{[\tau,t]}+P_{111}^{[s,\tau]}P_{111}^{[\tau,t]}$$
$$=2F(s,\tau)F(\tau,t)-{1\over 2}F(s,\tau)-{1\over 2}F(\tau,t)+{3\over 8}.$$
Denoting $G(s,t)=4F(s,t)-1$ one can rewrite this equation in the following form
$$G(s,t)={1\over 2}G(s,\tau)G(\tau,t).$$
The last equation has solutions $G(s,t)\equiv 0$, $G(s,t)={2\psi(t)\over \psi(s)}$ (for any $\psi\ne 0$) and
$$G(s,t)=\left\{\begin{array}{ll}
2, \ \ \mbox{if} \ \ 0\leq s<t\leq a\\[2mm]
0, \ \ \mbox{if} \ \ t> a.
\end{array}\right.$$
To these solutions correspond (by (\ref{u1})) the following QSPs of type $(12|a)$:
\begin{equation}\label{u11}
\M_3^{[s,t]}= \begin{pmatrix}
1/4 &1/4& \vline &1/4 &1/4\\[3mm]
1/4&1/4& \vline &1/4 &1/4
\end{pmatrix}.
\end{equation}
\begin{equation}\label{u12}
\M_4^{[s,t]}= \begin{pmatrix}
1/4 &1/4& \vline &1/4 &{1\over 4}+{\psi(t)\over 2\psi(s)}\\[3mm]
1/4&1/4& \vline &1/4 &{1\over 4}-{\psi(t)\over 2\psi(s)}
\end{pmatrix},
\end{equation}
where $\psi$ is such that $-{1\over 2}\leq \frac{\psi(t)}{\psi(s)}\leq {1\over 2}$.
This condition can be satisfied for a function $\psi$ when time is discrete, i.e., $t\in \mathbb N$.
Then for example we can take $\psi(n)=2^{-n}$.

The case $G(s,t)=2$ does not define a QSP, because in this case $F(s,t)>1/2$.

(b") For the case $g(s,t)={1\over 4}\left(\frac{\varphi(t)}{\varphi(s)}+1\right)$ (where $\varphi\ne 0$ is an arbitrary function) in (\ref{on})
the $(1, 2)$-stochastic cubic matrix has the form
\begin{equation}\label{eb4}
\M^{[s,t]}= \begin{pmatrix}
g(s,t)&g(s,t)& \vline &g(s,t) &L(s,t)\\[3mm]
{1\over 2}-g(s,t)&{1\over 2}-g(s,t)& \vline &{1\over 2}-g(s,t) &{1\over 2}-L(s,t)
\end{pmatrix},
\end{equation}
where $g(s,t)={1\over 4}\left(\frac{\varphi(t)}{\varphi(s)}+1\right)$ with $-1\leq \frac{\varphi(t)}{\varphi(s)}\leq 1$. In particular, this condition is satisfied for a positive and decreasing function $\varphi$. Here the function $L(s,t)$ should satisfy the following equation
\begin{equation}\label{q}
L(s,t)=2L(s,\tau)L(\tau,t)-{1\over 2}(L(s,\tau)+L(\tau,t))+{1\over 4}+{1\over 2}g(s,t).
\end{equation}

Note that the equation (\ref{q}) has solution
$L(s,t)=g(s,t)={1\over 4}\left(\frac{\varphi(t)}{\varphi(s)}+1\right)$.
We do not know any other solution of (\ref{q}).

Thus
\begin{equation}\label{jk}
\M_5^{[s,t]}= \begin{pmatrix}
g(s,t)&g(s,t)& \vline &g(s,t) &g(s,t)\\[3mm]
{1\over 2}-g(s,t)&{1\over 2}-g(s,t)& \vline &{1\over 2}-g(s,t) &{1\over 2}-g(s,t)
\end{pmatrix},
\end{equation}
where $g(s,t)={1\over 4}\left(\frac{\varphi(t)}{\varphi(s)}+1\right)$ with $-1\leq \frac{\varphi(t)}{\varphi(s)}\leq 1$ generates a
QSP of type $(12| a)$.

(c") In this case
$$g(s,t)=\left\{\begin{array}{ll}
1/2, \ \ \mbox{if} \ \ s\leq t<c\\[2mm]
1/4, \ \ \mbox{if} \ \  t\geq c.
\end{array}\right.$$
Then corresponding matrix is
\begin{equation}\label{sk}
\M_6^{[s,t]}=\left\{\begin{array}{ll}
\begin{pmatrix}
{1\over 2}&{1\over 2}& \vline &{1\over 2} &{1\over 2}\\[3mm]
0&0& \vline &0 &0
\end{pmatrix}, \ \ \mbox{if} \ \ s\leq t<c\\[5mm]
\begin{pmatrix}
{1\over 4}&{1\over 4}& \vline &{1\over 4} &{1\over 4}\\[3mm]
{1\over 4}&{1\over 4}& \vline &{1\over 4} &{1\over 4}
\end{pmatrix}, \ \ \mbox{if} \ \  t\geq c.
\end{array}
\right.
\end{equation}

Summarizing we have

\begin{pro} The matrices $\M_i^{[s,t]}$, $i=3,5,6$ defined above generate QSPs of type $(12| a)$.
The matrices $\M_4^{[n,m]}$, $n,m\in \mathbb N$, $n<m$ generate a
discrete-time QSP of type $(12| a)$.
\end{pro}

Take now the matrix $\mathbb Q^{[s,t]}=\left(\begin{array}{cc}
0& 0\\
1& 1\end{array}\right)$  then
\begin{equation}\label{gq}
g(s,t)=4g(s,\tau)g(\tau,t)-g(s,\tau).\end{equation}
It is easy to see that this equation has solution $g(s,t)\equiv 1/2$ (we do not know any other solution).
But this solution does not define a QSP, because from
$P_{000}^{[s,t]}+P_{010}^{[s,t]}=q^{[s,t]}_{00}=0$ it follows that $P_{000}^{[s,t]}=0\ne 1/2$.
%
%

\section{An example when Condition 1 is not satisfied}

In this section we consider an operation $a$ on $I=\{1,2, \dots, m\}$ which is not uniquely solvable.
Consider binary operation  $a(i,j)=\max\{i,j\}.$ It is not uniquely solvable, in general. Indeed, for $m\geq 2$, the equation
$\max\{x, m\}=m$ has many solutions: $x=1,2,\dots, m$.

Let $\sigma$ is a fixed stochasticity of cubic matrices then the QSP corresponding to $\max$ operation is denoted as type $(\sigma | \max)$.
Here we give some examples of such QSP.

For simplicity we take $m=2$ and solve the equation (\ref{KC}) for matrix $\M^{[s,t]}=\left(a_{ijk}^{[s,t]}\right)_{i,j,k=1}^{2}$.

In the case of multiplication corresponding to the binary operation  $a(i,j)=\max\{i,j\}$ the equation (\ref{KC}) is in the following form
\begin{equation}\label{kcm}
\left\lbrace
\begin{array}{llllllll}
a_{111}^{[s,t]}=a_{111}^{[s,\tau]}a_{111}^{[\tau,t]}+a_{112}^{[s,\tau]}a_{211}^{[\tau,t]}\\[2mm]
a_{112}^{[s,t]}=a_{111}^{[s,\tau]}a_{112}^{[\tau,t]}+a_{112}^{[s,\tau]}a_{212}^{[\tau,t]}\\[2mm]
a_{211}^{[s,t]}=a_{211}^{[s,\tau]}a_{111}^{[\tau,t]}+a_{212}^{[s,\tau]}a_{211}^{[\tau,t]}\\[2mm]
a_{212}^{[s,t]}=a_{211}^{[s,\tau]}a_{112}^{[\tau,t]}+a_{212}^{[s,\tau]}a_{212}^{[\tau,t]}\\[2mm]
a_{121}^{[s,t]}=a_{111}^{[s,\tau]}a_{121}^{[\tau,t]}+a_{112}^{[s,\tau]}a_{221}^{[\tau,t]}+
a_{122}^{[s,\tau]}a_{221}^{[\tau,t]}+a_{121}^{[s,\tau]}a_{111}^{[\tau,t]}+a_{121}^{[s,\tau]}a_{121}^{[\tau,t]}+a_{122}^{[s,\tau]}a_{211}^{[\tau,t]}\\[2mm]
a_{122}^{[s,t]}=a_{111}^{[s,\tau]}a_{122}^{[\tau,t]}+a_{121}^{[s,\tau]}a_{112}^{[\tau,t]}
+a_{121}^{[s,\tau]}a_{122}^{[\tau,t]}+a_{112}^{[s,\tau]}a_{222}^{[\tau,t]}+a_{122}^{[s,\tau]}a_{212}^{[\tau,t]}
+a_{122}^{[s,\tau]}a_{222}^{[\tau,t]}\\[2mm]
a_{221}^{[s,t]}=a_{211}^{[s,\tau]}a_{121}^{[\tau,t]}+a_{212}^{[s,\tau]}a_{221}^{[\tau,t]}+a_{221}^{[s,\tau]}a_{111}^{[\tau,t]}+a_{221}^{[s,\tau]}a_{121}^{[\tau,t]}+a_{222}^{[s,\tau]}a_{211}^{[\tau,t]}+a_{222}^{[s,\tau]}a_{221}^{[\tau,t]}\\[2mm]
a_{222}^{[s,t]}=a_{212}^{[s,\tau]}a_{222}^{[\tau,t]}+a_{221}^{[s,\tau]}a_{112}^{[\tau,t]}
+a_{221}^{[s,\tau]}a_{122}^{[\tau,t]}+a_{222}^{[s,\tau]}a_{212}^{[\tau,t]}+a_{222}^{[s,\tau]}a_{222}^{[\tau,t]}+a_{211}^{[s,\tau]}a_{122}^{[\tau,t]}
\end{array}\right.
\end{equation}

Denoting
\begin{equation}\label{b}
b_{ij}^{[s,t]}=a_{i1j}^{[s,t]}+a_{i2j}^{[s,t]}, \ \ B^{[s,t]}=\left(b_{ij}^{[s,t]}\right)
\end{equation} one can reduce the
system  (\ref{kcm}) to the following one

\begin{equation}\label{bb}
\left\lbrace
\begin{array}{llllllll}
b_{11}^{[s,t]}=b_{11}^{[s,\tau]}b_{11}^{[\tau,t]}+b_{12}^{[s,\tau]}b_{21}^{[\tau,t]}\\[2mm]
b_{12}^{[s,t]}=b_{11}^{[s,\tau]}b_{12}^{[\tau,t]}+b_{12}^{[s,\tau]}b_{22}^{[\tau,t]}\\[2mm]
b_{21}^{[s,t]}=b_{21}^{[s,\tau]}b_{11}^{[\tau,t]}+b_{22}^{[s,\tau]}b_{21}^{[\tau,t]}\\[2mm]
b_{22}^{[s,t]}=b_{21}^{[s,\tau]}b_{12}^{[\tau,t]}+b_{22}^{[s,\tau]}b_{22}^{[\tau,t]}.
\end{array}\right.
\end{equation}
Note that 1-4 equations of the system (\ref{kcm}) can be solved independently from 5-8 equations.
Therefore if we solve system of 1-4 equations of (\ref{kcm}) and solve  system (\ref{bb}) then by (\ref{b}) we can
find all unknown functions of (\ref{kcm}). Let us realize this argument.

Denote $c_{ij}^{[s,t]}=a_{i1j}^{[s,t]}$, $C^{[s,t]}=\left(c_{ij}^{[s,t]}\right)$ then 1-4 equations of the system (\ref{kcm}) is
\begin{equation}\label{cm}
\left\lbrace
\begin{array}{llllllll}
c_{11}^{[s,t]}=c_{11}^{[s,\tau]}c_{11}^{[\tau,t]}+c_{12}^{[s,\tau]}c_{21}^{[\tau,t]}\\[2mm]
c_{12}^{[s,t]}=c_{11}^{[s,\tau]}c_{12}^{[\tau,t]}+c_{12}^{[s,\tau]}c_{22}^{[\tau,t]}\\[2mm]
c_{21}^{[s,t]}=c_{21}^{[s,\tau]}c_{11}^{[\tau,t]}+c_{22}^{[s,\tau]}c_{21}^{[\tau,t]}\\[2mm]
c_{22}^{[s,t]}=c_{21}^{[s,\tau]}c_{12}^{[\tau,t]}+c_{22}^{[s,\tau]}c_{22}^{[\tau,t]}.
\end{array}\right.
\end{equation}

Both system of equations (\ref{bb}) and (\ref{cm}) are Kolmogorov-Chapman equations for square matrices.
 Using known solutions for these equations, (for example, $\mathbb Q_i$, $i=1,2,...,7$ introduced in the previous section) one can give
 concrete solutions of the system (\ref{kcm}).
 Namely, if $B^{[s,t]}=\left(b_{ij}^{[s,t]}\right)$ is a solution to (\ref{bb}) and $C^{[s,t]}=\left(c_{ij}^{[s,t]}\right)$ is a solution to (\ref{cm})
 then corresponding solution to the system (\ref{kcm}) is

 \begin{equation}\label{md}
\M_7^{[s,t]}= \begin{pmatrix}
c_{11}^{[s,t]} &c_{12}^{[s,t]}& \vline &c_{21}^{[s,t]} &c_{22}^{[s,t]}\\[3mm]
b^{[s,t]}_{11}-c_{11}^{[s,t]} &b^{[s,t]}_{12}-c_{12}^{[s,t]}& \vline &b^{[s,t]}_{21}-c_{21}^{[s,t]} &b^{[s,t]}_{22}-c_{22}^{[s,t]}
\end{pmatrix}.
\end{equation}
\begin{thm}\label{tm}  Let $B^{[s,t]}=\left(b_{ij}^{[s,t]}\right)$ be a solution to (\ref{bb}) and
 $C^{[s,t]}=\left(c_{ij}^{[s,t]}\right)$ be a solution to (\ref{cm}) with $c_{ij}^{[s,t]}\in [0,1]$
 and $b^{[s,t]}_{ij}-c_{ij}^{[s,t]}\in [0,1]$ for any $i,j=1,2$, $0\leq s<t$ then the family
 of matrices $\M_7^{[s,t]}$ given in (\ref{md}) is a QSP of type
\begin{itemize}
\item[-] $(12 | \max)$ iff $B^{[s,t]}$ is left stochastic for any $0\leq s<t$.

\item[-] $(13 | \max)$ iff $B^{[s,t]}$ (resp. $C^{[s,t]}$) with non negative elements with sum of all elements equals to 2 (resp. 1).

\item[-] $(23 | \max)$ iff $B^{[s,t]}$ is right stochastic for any $0\leq s<t$.

\item[-] $(1 | \max)$ iff $B^{[s,t]}$ (resp. $C^{[s,t]}$) with non negative elements with sum of all elements of each column equals to 2 (resp. left stochastic).

 \item[-] $(2 | \max)$ never.

\item[-] $(3 | \max)$ iff $B^{[s,t]}$ (resp. $C^{[s,t]}$) with non negative elements with sum of all elements of each row equals to 2 (resp. right stochastic).
\end{itemize}
\end{thm}
\begin{proof} All types (expect the type $(2 | \max)$) follow from the definitions of the corresponding
stochasticity. In the case  $(2 | \max)$ it is necessary that $b^{[s,t]}_{ij}\equiv 1$, but it is easy to see that
such quadratic matrix $B^{[s,t]}$ does not satisfy equation (\ref{bb}).
\end{proof}
\section{Dynamical systems of QSPs}

For QSPs generated by $\M^{[s,t]}_i$, $i=1,\dots,7$ using (\ref{ks}), \eqref{xt12}, let us give the time
behavior of the distribution $x^{(t)}=(x_0^{(t)}, x_1^{(t)})\in S^1$.
Fix $s\geq 0$
and by take a vector $x^{(s)}=(x_0^{(s)}, x_1^{(s)})\in S^1$.

{\bf Case $\M^{[s,t]}_1$ and  $\M^{[s,t]}_3$.}  By formula \eqref{xt12}
independently on the vector $x^{(s)}$, for \textsl{any} $t>s$,  we get
$$
x_0^{(t)} =\frac{1}{2}, \ \ x_1^{(t)} =\frac{1}{2}.
$$
Thus the time behavior of $x^{(t)}$ is clear: start process at time $s$ with an arbitrary initial
distribution vector $x^{(s)}$ then as soon as
the time $t$ turns on the distribution of the system goes to
the distribution $(1/2, 1/2)$ and this distribution remains stable
during all time $t>s$.

{\bf Case $\M^{[s,t]}_2$.}  By formula \eqref{xt12}, for fixed $s\geq 0$, given
vector $x^{(s)}$ and \textsl{any} $t>s$,  we get
\begin{align*}
x_0^{(t)}& =\left(\frac{1}{2}+{\Phi(t)\over 4\Phi(s)}\right)x_0^{(s)}+\left(\frac{1}{2}-{\Phi(t)\over 4\Phi(s)}\right)x_1^{(s)},\\[2mm]
x_1^{(t)}& =\left(\frac{1}{2}-{\Phi(t)\over 4\Phi(s)}\right)x_0^{(s)}+\left(\frac{1}{2}+{\Phi(t)\over 4\Phi(s)}\right)x_1^{(s)}.
\end{align*}
The time behavior of $x^{(t)}$ depends on function $\Phi$ (which by our assumption satisfies $-1/3\leq \Phi(t)/\Phi(s)\leq 1/3$).
If for example, $\Phi$ is such that
\begin{equation}\label{ub}\lim_{t-s\to\infty}{\Phi(t)\over 4\Phi(s)}=\omega, \ \ \mbox{with} \ \ \omega\in [-{1\over 12}, {1\over 12}].\end{equation}
Then
\begin{align*}
\lim_{t\to\infty}x_0^{(t)}& =\left(\frac{1}{2}+\omega\right)x_0^{(s)}+\left(\frac{1}{2}-\omega\right)x_1^{(s)},\\[2mm]
\lim_{t\to\infty}x_1^{(t)}& =\left(\frac{1}{2}-\omega\right)x_0^{(s)}+\left(\frac{1}{2}+\omega\right)x_1^{(s)}.
\end{align*}
In case when the limit (\ref{ub}) does not exists then limit of $x^{(t)}$ does not exist too.

{\bf Case $\M^{[s,t]}_4$.}  In this case  we have
\begin{align*}
x_0^{(t)}& =\frac{1}{2}x_0^{(s)}+\left(\frac{1}{2}+{\psi(t)\over 4\psi(s)}\right)x_1^{(s)},\\[2mm]
x_1^{(t)}& =\frac{1}{2}x_0^{(s)}+\left(\frac{1}{2}-{\psi(t)\over 4\psi(s)}\right)x_1^{(s)}.
\end{align*}
As previous case, the time behavior of $x^{(t)}$ depends on function $\phi$ (which by our assumption satisfies $-1/2\leq \phi(t)/\phi(s)\leq 1/2$).

{\bf Case $\M^{[s,t]}_5$.}  In this case independently on the initial state vector $x^{(s)}$ we obtain
\begin{align*}
x_0^{(t)}& =\frac{1}{2}+{\varphi(t)\over 4\varphi(s)},\\[2mm]
x_1^{(t)}& =\frac{1}{2}-{\varphi(t)\over 4\varphi(s)}.
\end{align*}
This is an interesting dynamical system, because at each initial (fixed) time $s$ the system does not depend on the initial state $x^{(s)}$ of the system.
  The trajectory only depends on the initial time itself and the time behavior of $x^{(t)}$ depends on function $\varphi$ (which by our assumption satisfies $-1\leq \varphi(t)/\varphi(s)\leq 1$).

  {\bf Case $\M^{[s,t]}_6$.}  In this case independently on the initial state vector $x^{(s)}$ we obtain
$$
x_0^{(t)} =1-x_1^{(t)}=\left\{\begin{array}{ll}
\frac{3}{4}, \ \ \mbox{if} \ \ 0\leq s<t<a\\[2mm]
\frac{1}{2}, \ \ \mbox{if} \ \ t\geq a.
\end{array}\right.$$
Thus we get a discontinuous (with respect to time) dynamical system, the trajectory has limit $1/2$.

{\bf Case $\M^{[s,t]}_7$.} Consider QSP of type $(3 | \max)$ (other cases can be considered similarly). By (\ref{ks}) and Theorem \ref{tm}
we get
  \begin{align*}
x_0^{(t)}& =c_{11}^{[s,t]}(x_0^{(s)})^2+\left(b_{11}^{[s,t]}-c_{11}^{[s,t]}+c_{21}^{[s,t]}\right)x_0^{(s)}x_1^{(s)}
+\left(b_{21}^{[s,t]}-c_{21}^{[s,t]}\right)(x_1^{(s)})^2,\\[2mm]
x_1^{(t)}& =c_{12}^{[s,t]}(x_0^{(s)})^2+\left(b_{12}^{[s,t]}-c_{12}^{[s,t]}+c_{22}^{[s,t]}\right)x_0^{(s)}x_1^{(s)}
+\left(b_{22}^{[s,t]}-c_{22}^{[s,t]}\right)(x_1^{(s)})^2.
\end{align*}

   This is a quadratic continuous time dynamical system. The behavior of $x^{(t)}$ depends on the matrix $\M^{[s,t]}_7$.
   One can choose this matrix to make the behavior of the dynamical system as reach as needed
   (see \cite{GMR}, \cite{J}, \cite{ly}, \cite{MG} for some examples of quadratic dynamical systems and their applications).
%

\end{document}